\documentclass[a4paper,12pt]{article}
\usepackage{amsmath,amsthm,amssymb}
\usepackage{tabularx}
\newcommand{\GF}{{\mathbb F}}

\newtheorem{Thm}{Theorem}[section]
\newtheorem{Lem}[Thm]{Lemma}

\theoremstyle{definition}

\begin{document}

\title{{On the support designs of extremal binary doubly even 
self-dual codes
}
\footnote{
This work was supported by JSPS KAKENHI Grant Number 22840003, 24740031.
}
}

\author{
Naoyuki Horiguchi\thanks{Graduate School of Science, Chiba University, Chiba 263-8522, Japan},
Tsuyoshi Miezaki \thanks{
Faculty of Education, Art and Science, 
Yamagata University, 
Yamagata 990-8560, Japan
}
and Hiroyuki Nakasora \thanks{Graduate School of Natural Science and Technology, Okayama University, Okayama 700-8530, Japan}}

\date{}

\maketitle

\paragraph{Abstract.}
Let $D$ be the support design of the minimum weight of an extremal binary doubly even self-dual $[24m,12m,4m+4]$ code.
In this note, we consider the case when $D$ becomes a $t$-design with $t \geq 6$.












\setcounter{section}{+0}
\section{Introduction}

Let $C$ be an extremal binary doubly even self-dual $[24m,12m,4m+4]$ code. 
Known examples of $C$ are the extended Golay code $\mathcal{G}_{24}$ and the extended quadratic residue code of length $48$. 
It was shown by Zhang \cite{Zhang(1999)} that $C$ does not exist if $m \geq 154$. 
The \textit{support} of a codeword $c=(c_{1}, \dots, c_{24m}) \in C$, 
$c_{i} \in \GF_{2}$ is 
the set of indices of its nonzero coordinates: $supp (c) = \{ i : c_{i} \neq 0 \}$. 
The \textit{support design} of $C$ for a given nonzero weight $w$ is the design 
for which the points are the $24m$ coordinate indices, and the blocks are the supports of all codewords of weight $w$. 
Let $D_{w}$ be the support design of $C$ for any $w \equiv 0 \pmod 4$ with $4m+4 \leq w \leq 24m-(4m+4)$.
Then it is known from the Assmus-Mattson theorem \cite{assmus-mattson} 
that $D_{w}$ becomes a $5$-design. 

One of the most interesting questions around the Assmus-Mattson theorem is the following: \\
``Can any $D_{w}$ become a $6$-design?"
Note that no 6-design has yet been obtained by applying the Assmus-Mattson theorem (see \cite{Bannai-Koike-Shinohara-Tagami}).

In this note, we consider the support design $D(=D_{4m+4})$ of the minimum weight of $C$. 
By the Assmus-Mattson theorem, $D$ is a $5$-$(24m,4m+4,\binom{5m-2}{m-1})$ design. 
Suppose that $D$ is a $t$-$(24m,4m+4,\lambda_{t})$ design with $t \geq 6$. 
It is easily seen that $\lambda_{t}= \binom{5m-2}{m-1} \binom{4m-1}{t-5}/\binom{24m-5}{t-5}$ is a nonnegative integer. 
It is known that  if $D$ is a $6$-design, then it is a $7$-design 
by a strengthening of the Assmus-Mattson theorem \cite{strengthening of the Assmus-Mattson theorem}.
In \cite[Theorem 5]{Bannai-Koike-Shinohara-Tagami}, Bannai et al. showed that $D$ is not a $6$-design 
if $m$ is ($\leq 153$) not in the set 
$\{8,15,19,35,40,41,42,50,51,52,55,57,59,60,63,
65,74,75,76,80,86,90,93,100,101,104,105,107,\\
118,125,127,129,130,135,143,144,150,151\}.$
It was also shown in \cite{Bannai-Koike-Shinohara-Tagami} that if $D$ is an $8$-design, then $m$ must be in the set $\{8,42,63,75,130 \}$, 
and $D$ is never a $9$-design.

In this note, we extend a method used in \cite{Bannai-Koike-Shinohara-Tagami}. First, we will prepare it in Section 2.
In Section 3, we consider a self-orthogonal $7$-design which has parameters equal to those of $D$. 
Then we will show that there is no self-orthogonal $7$-design for some $m$. 
For the remainder $m$, we consider the support $7$- and $8$-design of $C$ in Section 4. 
In summary, our main result is the following theorem.

\begin{Thm}\label{thm:main thm}
Let $D$ be the support $t$-design of the minimum weight of an extremal binary doubly even self-dual $[24m,12m,4m+4]$ code ($m \leq 153$). 
If $t \geq 6$, then $D$ is a $7$-design and $m$ must be in the set 
$\{15,52,55,57,59,60,63,90,93,104,105,107, 118,125,127,135,143,151 \}$, 
and $D$ is never an $8$-design.

\end{Thm}




\section{Preparation}\label{:section 2}

A \textit{$t$-$(v,k,{\lambda})$ design} is a pair 
$\mathcal{D}=(X,\mathcal{B})$, where $X$ is a set of points of 
cardinality $v$, and $\mathcal{B}$ a collection of $k$-element subsets
of $X$ called blocks, with the property that any $t$ points are 
contained in precisely $\lambda$ blocks.
It follows that every
$i$-subset of points $(i \leq t)$ is contained in exactly
$\lambda_{i}= \lambda \binom{v-i}{t-i} / \binom{k-i}{t-i}$ blocks. 

Let $C$ be an extremal binary doubly even self-dual $[24m,12m,4m+4]$ code. 
Let $D$ be the support design of the minimum weight of $C$.
Then $D$ is a $5$-$(24m,4m+4,\binom{5m-2}{m-1})$ design. 
Suppose that $D$ is a $t$-$(24m,4m+4,\lambda_{t})$ design with $t \geq 6$, 
where $\lambda_{t}= \binom{5m-2}{m-1} \binom{4m-1}{t-5}/\binom{24m-5}{t-5}$. 
It is known, by a strengthening of the Assmus-Mattson theorem \cite{strengthening of the Assmus-Mattson theorem}, 
that if $D$ is a $6$-design, then it is also a $7$-design. 
Hence $\lambda_{6}$ and $\lambda_{7}$ are nonnegative integers. 
Then, by computations, we have the following lemma: 
\begin{Lem}\label{lem:6,7}
For $1\leq m\leq 153$, the values $\lambda_{6}$ and 
$\lambda_{7}$
are both nonnegative integers only if $m \in M = 
\{5,8,15,19,35,40,41,42,50,51,52,55,57,59,60,63,65,74,75,76,80,86,90,\\
93, 100,101,104,105,107,118,125,127,129,130,135,143,144,150,151 \}$.
\end{Lem}

The \textit{Stirling numbers of the second kind} $S(n, k)$ (see \cite{Weisstein:Stirling Number of the Second Kind})
are the number of ways to partition a set of $n$ elements into $k$ nonempty subsets.
The Stirling numbers of the second kind can be computed from the sum
$$S(n,k)= \frac{1}{k!} \sum_{i=0}^{k} (-1)^{i} \binom{k}{i} (k-i)^{n}$$
or the generating function 
\begin{align}
x^{n}  =\sum_{m=0}^{n}S(n,m) (x)_{m} \label{eqn:def of stirling num 2nd}
\end{align}
where $(x)_{m}=x(x-1)\cdots (x-m+1)$. 
Special cases include $S(n,0)= \delta_{n,0}$ and  $S(0,k)= \delta_{0,k}$, 
where $\delta_{n,k}$ is the \textit{Kronecker delta}.

We can easily obtain a table of the initial Stirling numbers of the second kind:

\begin{tabular}{c|cccccccccc}   \label{tab:Stirling numbers of the second kind}
   $n \backslash k$ & 0 & 1 & 2 & 3 & 4 & 5 & 6 & 7 & $8$& $\cdots$ \\ \hline
   1     & 0 & 1 &  &  &  & & & \\
   2     & 0 & 1 & 1 &  &  &  & & \\
   3     & 0 & 1 & 3 & 1  &  &  & & & \\
   4     & 0 & 1 & 7 & 6 & 1 &  & & \\
   5     & 0 & 1 & 15 & 25 & 10 & 1 & & \\
   6     & 0 & 1 & 31 & 90 & 65 & 15 & 1 & \\
   7     & 0 & 1 & 63 & 301 & 350 & 140 &21 & 1 \\
   8     & 0 & 1 & 127 & 966 & 1701 & 1050 & 266 & 28 & 1 \\
   $\vdots$     &    &  &  &  &  & & 
\end{tabular}

\bigskip

For a block $B$ of the design $\mathcal{D}$, let $n_{i}$ 
be the number of blocks of $\mathcal{B}$ that meet $B$ in $i$ points, 
where $0 \leq i \leq k$.
We set                                     
$A_{s}= \sum_{i=0}^{k} (i)_{s} n_{i}$ for $0 \leq s \leq t$. 

By (\ref{eqn:def of stirling num 2nd}), for $0 \leq t' \leq t$, we have
\begin{eqnarray}
\sum_{i=0}^{k} i^{t'} n_{i} 
=\sum_{i=0}^{k} \left\{ \sum_{h=0}^{t'} S(t',h) (i)_{h} \right\} n_{i} 
= \sum_{h=0}^{t'} S(t',h) A_{h}. \label{eqn:i^t'=S(t',h)}
\end{eqnarray}  

The \textit{elementary symmetric polynomials} in $n$ variables $x_{1}, x_{2}, \ldots, x_{n}$ 
written $\sigma_{k,n}$ \\ for $k=0,1, \ldots, n$, can be defined as 
$\sigma_{0,n} =1$ and 
$$\sigma_{k,n} =\sigma_{k}(x_{1}, x_{2}, \ldots, x_{n}) = \sum_{\substack{1 \leq w_{1} < \cdots < w_{k} \leq n}} 
x_{w_{1}} x_{w_{2}} \cdots x_{w_{k}}.$$

Then we have the following lemma.

\begin{Lem} \label{lem:single block intersection}

If $l \leq t$,

$$\sum_{i=0}^{k} (i-x_{1}) (i-x_{2}) \cdots (i-x_{l})n_{i} 
 =  \sum_{\theta =0}^{l}   (-1)^{\theta}  \sigma_{\theta,l} \left( \sum_{h =0}^{l-\theta} S(l- \theta, h)  A_{h}  \right).$$

\end{Lem}

\begin{proof}
\begin{align*}
& \sum_{i=0}^{k} (i-x_{1}) (i-x_{2}) \cdots (i-x_{l})n_{i} \\
& =\sum_{i=0}^{k} \left\{ i^{l}- \sum_{1 \leq w_{1} \leq l} x_{w_{1}} i^{l-1}  
+\sum_{1 \leq w_{1} < w_{2} \leq l} x_{w_{1}} x_{w_{2}} i^{l-2} + \cdots + (-1)^{l} x_{w_{1}} x_{w_{2}} \cdots x_{w_{l}} \right\} n_{i} \\
& =\sum_{i=0}^{k} \left\{ \sum_{\theta =0}^{l} (-1)^{\theta} \sigma_{\theta,l} i^{l- \theta} \right\} n_{i} \\
& =\sum_{\theta=0}^{l} (-1)^{\theta} \sigma_{\theta,l} 
\sum_{i=0}^{k}  i^{l- \theta}  n_{i} \\
&\text{by (\ref{eqn:i^t'=S(t',h)})}, \\
& =  \sum_{\theta =0}^{l}   (-1)^{\theta}  \sigma_{\theta,l} \left( \sum_{h =0}^{l-\theta} S(l- \theta, h)  A_{h}  \right).
\end{align*}
  
\end{proof}

\section{On the nonexistence of some self-orthogonal $7$-designs}\label{:section 3}

A $t$-$(v,k, \lambda)$ design is called \textit{self-orthogonal} 
if the intersection of any two blocks of the design has 
the same parity as the block size $k$ \cite{Tonchev2}. 
Notice that a design obtained from the supports of 
the minimum weight codewords in a binary doubly even code 
is self-orthogonal as the blocks have lengths 
a multiple of $4$ and the overlap of supports of 
any two codewords must have even size.

In this section, let $D'=(X',\mathcal{B'})$ be a self-orthogonal $7$-$(24m,4m+4,\lambda_{7})$ design, 
where $\lambda_{7} =  \binom{5m-2}{m-1} \frac{(4m-1)(4m-2)}{(24m-5)(24m-6)}$ and $m \in M$.
For a block $B$ of the design $D'$, let $n_{i}$ 
be the number of blocks of $\mathcal{B}'$ that meet $B$ in $i$ points, 
where $0 \leq i \leq 4m+4$, 
and $n_i=0$ if $i$ is odd (since $D'$ is self-orthogonal).

Using the fundamental equation in 
Koch \cite{koch} and also 
\cite[Proof of Theorem 5]{Bannai-Koike-Shinohara-Tagami}, we have
\[
\sum_{i=0}^{4m+4}\binom{i}{s}n_i=\binom{4m+4}{s}\lambda_s
\]
for $0 \leq s \leq t$. Therefore we set 
\begin{align}\label{eqn:A_{s}}
A_{s}= \sum_{i=0}^{4m+4} (i)_{s} n_{i}= (4m+4)_{s} \lambda_{s} 
\end{align}
for $0 \leq s \leq 7$. 

For the design $D'$, we define 
 $$F(m,4m+4;[ x_{1}, x_{2}, x_{3}, x_{4}, x_{5},x_{6},x_{7}]) = \sum_{i=0}^{4m+4} (i-x_{1}) (i-x_{2}) \cdots (i-x_{7})n_{i}.$$

By Lemma \ref{lem:single block intersection}, we have

\begin{align}
 &F(m,4m+4;[ x_{1}, x_{2}, x_{3}, x_{4}, x_{5},x_{6},x_{7}]) 
 =\sum_{\theta =0}^{7}   (-1)^{\theta}  \sigma_{\theta,7} \left( \sum_{h =0}^{7-\theta} S(7- \theta, h)  A_{h}  \right) \notag \\
 &= -\sigma{}_{7,7} A_{0} 
                 +(\sigma{}_{6,7}-\sigma{}_{5,7}+\sigma{}_{4,7}-\sigma{}_{3,7}+\sigma{}_{2,7}-\sigma{}_{1,7}+1) A_{1} \notag \\
                & +(-\sigma{}_{5,7}+3\sigma{}_{4,7}-7\sigma{}_{3,7}+15\sigma{}_{2,7}-31\sigma{}_{1,7}+63)  A_{2}
                 +(\sigma{}_{4,7}-6\sigma{}_{3,7}+25\sigma{}_{2,7}-90\sigma{}_{1,7}+301)  A_{3} \notag \\
                & +(-\sigma{}_{3,7}+10\sigma{}_{2,7}-65\sigma{}_{1,7}+350)  A_{4}
                 +(\sigma{}_{2,7}-15\sigma{}_{1,7}+140)  A_{5}
                 +(-\sigma{}_{1,7}+21)  A_{6}
                 + A_{7}.\label{eq:F(m,4m+4)}
\end{align}

Then we have the following theorem.

\begin{Thm}\label{thm:so 7-design}
If $m \in \{ 8,40,42,50,74,76,80,86,100,130,144,150 \}$, 
then there is no self-orthogonal $7$-$(24m,4m+4,\lambda_{7})$ design.

\end{Thm}

\begin{proof}

By (\ref{eq:F(m,4m+4)}), we have
\begin{align}
& F(m,4m+4; [0, 2, 4, 6, 8, 10, 12]) \notag \\
& =\sum_{i=0}^{4m+4} i(i-2)(i-4)(i-6)(i-8)(i-10)(i-12) n_{i} \label{eq:F(m,4m+4),n_14}\\
& =10395A_{1}-10395A_{2}+4725A_{3}-1260A_{4}+210A_{5}-21A_{6}+A_{7} \label{eq:F(m,4m+4),A_7}.
\end{align}

By (\ref{eq:F(m,4m+4),n_14}), we have 
$$n_{14}=\frac{F(m,4m+4; [0, 2, 4, 6, 8, 10, 12])}{645120}-8n_{16}-36n_{18}- \cdots -\binom{2m+2}{7}n_{4m+4}.$$

By (\ref{eqn:A_{s}}) and (\ref{eq:F(m,4m+4),A_7}), we compute  the values of 
$\frac{F(m,4m+4; [0, 2, 4, 6, 8, 10, 12])}{645120}$ for $m \in M$ by using Magma \cite{MAGMA}. 
Then, if $m \in \{ 8,40,42,50,74,76,80,86,100,130,144,150 \}$, 
we have that $\frac{F(m,4m+4; [0, 2, 4, 6, 8, 10, 12])}{645120}$ is not an integer as in Table~\ref{tab:F(m)/645120}. 
Hence $n_{14}$ is not an integer, which is a contradiction. 
Therefore, there is no self-orthogonal $7$-$(24m,4m+4,\lambda_{7})$ design.

\begin{table}[h] 
\caption{$F(m,4m+4; [0, 2, 4, 6, 8, 10, 12])/645120$}
\begin{center}
\begin{tabular}{|l|l|}\hline  \label{tab:F(m)/645120}
$m$ & $F(m,4m+4; [0, 2, 4, 6, 8, 10, 12])$ \\ \hline \hline
8 & 1569595833/8  \\ \hline
40 & 69722676263111828528771787666297086782790166251961/4  \\ \hline
42 & 7413717557642396579773804378982932177616748595565925/2  \\ \hline
50 & 51322358900999864497776773019002155555828915534612183899278199/8  \\ \hline
74 & 48264250867613004754712114323888323106391063440390013274122385587 \\ 
& 249531646783997099951125/4 \\ \hline
76 & 17297525223319023619606726106935051432172166325470799567243756674 \\
& 224986310548399405901291385/8 \\ \hline
80 & 27359269561632513491639341643828473595091499699382750775301196495 \\
& 5421000133137628599208064366065/4 \\ \hline
86 & 744177567256730099369876802080305538893055728577932239814964847527 \\
& 166234593936158833143382070818871425/2 \\ \hline

100 & 136128986840732000501396957664485117364432305019915914548757138545 \\
& 97559161850934989081846304372757198545374868637953809/8 \\ \hline

130 & 318970748043555317972217445610203126388833094365347676907996644311 \\ 
    & 283294844062616913946522213194960593445240713768468992990519455574 \\
& 30274425744716671055/8 \\ \hline

144 & 103507027242828427789085756140225158658851288631068123683963832873 \\ 
    & 230423948691159118858639575450894139827154262422968053791406366132 \\
& 850642263742190215236965412560306275/8 \\ \hline

150 & 225179137631450932254612557265887240264932451320353248639071269342 \\
    &142043892423125433482640303802028438247793256844429544681810502004 \\
& 849575476006319668579289958493743350895465/4 \\ \hline            
\end{tabular}
\end{center}
\end{table}

\end{proof}

\section{On the nonexistence of some support $7$-designs}\label{:section 4}

In this section, let $D$ be the support design of the minimum weight $4m+4$ of an extremal binary doubly even $[24m,12m,4m+4]$ code $C$.
If $D$ is a $7$-design, the parameters are $(24m,4m+4,\lambda_{7})$, 
where $\lambda_{7} =  \binom{5m-2}{m-1} \frac{(4m-1)(4m-2)}{(24m-5)(24m-6)}$, and $m \in M$.
By Lemma \ref{lem:6,7} and Theorem \ref{thm:so 7-design}, we consider the remainder

\begin{align*}
m \in  \{5,15,19,35,41,51,52,55,57,59,60,63,65,75,90,93, & \\
 101,104,105,107,118,125,127,129,135,143,151\}.&
\end{align*}

Let $C_{u}$ be the set of all codewords of weight $u$ of $C$, where $4m+4 \leq u \leq 20m-4$. 
Fix $a \in C_{u}$ and define $n_{j}^{u}:= \big| \{ c \in C_{4m+4}: |supp (a) \cap supp (c)|=j \} \big|$. 
Then by using the fundamental equations in \cite[Proof of Theorem 5]{Bannai-Koike-Shinohara-Tagami},   
we have 
$$\sum_{j=0}^{4m+4} \binom{j}{s} n_{j}^u= \binom{u}{s}\lambda_{s}$$
for $s=0,1, \ldots, t$.  (Note that $n_j^u=0$ if $j$ is odd, and $n_{j}^u=0$ if $j > \frac{u}{2}$.) 
These equations have been studied for a $t$-design and some applications,  
e.g. see Cameron and van Lint \cite{cameron-lint}, Koch \cite{koch} and Tonchev \cite{tonchev-text1988}. 
We set                                      
$A_{s}^{u}= \sum_{j=0}^{4m+4} (j)_{s} n_{j}^u= (u)_{s} \lambda_{s}$ for $0 \leq s \leq 7$.

For the design $D$, we define 
 $$F(m,u;[ x_{1}, x_{2}, x_{3}, x_{4}, x_{5},x_{6},x_{7}]) = \sum_{j=0}^{4m+4} (j-x_{1}) (j-x_{2}) \cdots (j-x_{7})n_{i}^u.$$
Then as a generalization of equation (\ref{eq:F(m,4m+4)}), we have 

\begin{align}
 &F(m,u;[ x_{1}, x_{2}, x_{3}, x_{4}, x_{5},x_{6},x_{7}]) 
 =\sum_{\theta =0}^{7}   (-1)^{\theta}  \sigma_{\theta,7} \left( \sum_{h =0}^{7-\theta} S(7- \theta, h)  A_{h}^{u}  \right) \notag \\
 &= -\sigma{}_{7,7} A_{0}^{u} 
                 +(\sigma{}_{6,7}-\sigma{}_{5,7}+\sigma{}_{4,7}-\sigma{}_{3,7}+\sigma{}_{2,7}-\sigma{}_{1,7}+1) A_{1}^{u} \notag \\
                & +(-\sigma{}_{5,7}+3\sigma{}_{4,7}-7\sigma{}_{3,7}+15\sigma{}_{2,7}-31\sigma{}_{1,7}+63)  A_{2}^{u}
                 +(\sigma{}_{4,7}-6\sigma{}_{3,7}+25\sigma{}_{2,7}-90\sigma{}_{1,7}+301)  A_{3}^{u} \notag \\
                & +(-\sigma{}_{3,7}+10\sigma{}_{2,7}-65\sigma{}_{1,7}+350)  A_{4}^{u}
                 +(\sigma{}_{2,7}-15\sigma{}_{1,7}+140)  A_{5}^{u}
                 +(-\sigma{}_{1,7}+21)  A_{6}^{u}
                 + A_{7}^{u}.\label{eq:F(m,u)}
\end{align}

Then we give the following theorem. (Note that Theorem \ref{thm:so 7-design} is a stronger result.)

\begin{Thm}\label{thm:support 7-design}
If $m \in \{5,19,35,41,51,65,75,101,129 \}$, 
then $D$ (the support design of the minimum weight of $C$) 
is not a $7$-design.


\end{Thm}

\begin{proof}
By (\ref{eq:F(m,u)}), we have
\begin{align*}
& F(m,4m+8; [0, 2, 4, 6, 8, 10, 12]) \\
& =\sum_{j=0}^{4m+4} j(j-2)(j-4)(j-6)(j-8)(j-10)(j-12) n_{j}^{4m+8} \\
& =10395A_{1}^{4m+8}-10395A_{2}^{4m+8}+4725A_{3}^{4m+8}-1260A_{4}^{4m+8}+210A_{5}^{4m+8}-21A_{6}^{4m+8}+A_{7}^{4m+8}.
\end{align*}

Then we have 
$$n_{14}^{4m+8}	=\frac{F(m,4m+8; [0, 2, 4, 6, 8, 10, 12])}{645120}-8n_{16}^{4m+8}-36n_{18}^{4m+8}- \cdots -\binom{2m+2}{7}n_{4m+4}^{4m+8}.$$

If $m \in \{5,19,35,41,51,65,75,101,129 \}$,
by computation using Magma,
we have that \\ $\frac{F(m,4m+8; [0, 2, 4, 6, 8, 10, 12])}{645120}$ is not an integer as in Table~\ref{tab:F(m,4m+8)/645120}. 
Hence $n_{14}^{4m+8}$ is not an integer, which is a contradiction. 
Therefore, $D$ is not a $7$-design.

\begin{table}[h] 
\caption{$F(m,4m+8; [0, 2, 4, 6, 8, 10, 12])/645120$}
\begin{center}
\begin{tabular}{|l|l|}\hline  \label{tab:F(m,4m+8)/645120}
$m$ & $F(m,4m+8; [0, 2, 4, 6, 8, 10, 12])/645120$ \\ \hline \hline
5 & 9009/4  \\ \hline
19 & 10290542185356908976248643/8  \\ \hline
35 & 240192525434759794880275676371011296919815805/8  \\ \hline
41 & 1229066981776753671012029436288037892461385328646335/4  \\ \hline
51 & 836449644579567992045815972312879647652910128602615298771389885/8 \\  \hline
65 & 72975174207654767982109272917411685745718438510666139598156100263 \\
& 01949797750545/8 \\ \hline
75 & 71317588499310631419430590525991955846021452139909800919361087401 \\
& 3324199822838428254310609/4 \\ \hline
101 & 95541360721321819333355415268808168206345704645344007828295378667 \\
& 333445530368924972096458449177337659658397691862895305/4 \\ \hline

129 & 26230778791143794560883418189575439901696761060680966584544120210 \\
& 69910259333223193114235770446444499518305259605168488333726043587 \\ 
& 913132060022892602625/8 \\ \hline
       
\end{tabular}
\end{center}
\end{table}

\end{proof}

Finally, we consider when $D$ is an $8$-design. 
From Lemma \ref{lem:6,7}, 
if $\lambda_{8}$ is also an integer, we have $m \in \{8,42,63,75,130 \}.$ 
By Theorem \ref{thm:so 7-design} and \ref{thm:support 7-design}, 
we have only the remainder $m=63$.
Let $D''$ be a self-orthogonal $8$-$(24m,4m+4,\lambda_{8})$ design, 
where $\lambda_{8} =  \binom{5m-2}{m-1} \frac{(4m-1)(4m-2)(4m-3)}{(24m-5)(24m-6)(24m-7)}$.
With $n_i$ as defined in Section \ref{:section 3}, 
which equals $n_i^{4m+4}$, we set 
$A_{s}= \sum_{i=0}^{4m+4} (i)_{s} n_{i}= (4m+4)_{s} \lambda_{s}$ for $0 \leq s \leq 8$. 
For the design $D''$, 
we have

\begin{align*}
 F(m,4m+4;[ x_{1}, x_{2}, x_{3}, x_{4}, x_{5},x_{6},x_{8}]) &= \sum_{i=0}^{4m+4} (i-x_{1}) (i-x_{2}) \cdots (i-x_{8})n_{i} \\
 &=\sum_{\theta =0}^{8}   (-1)^{\theta}  \sigma_{\theta,8} \left( \sum_{h =0}^{8-\theta} S(8- \theta, h)  A_{h}  \right). \notag \\
\end{align*}

Then, we have 
$$n_{16}=\frac{F(m,4m+4; [0, 2, 4, 6, 8, 10, 12,14])}{10321920}-9n_{18}-45n_{20}- \cdots -\binom{2m+2}{8}n_{4m+4}.$$

In the case $m=63$, 
by a computation using Magma, 
we have 

\begin{align*}
& \frac{F(63,4 \cdot 63+4; [0, 2, 4, 6, 8, 10, 12,14])}{10321920} \\
&\hspace{-2mm} =-16809515472136742134534321134853418244406436165053567105402493489903309445518999/1792.
\end{align*}
Hence $n_{16}$ is not an integer. Therefore, if $m=63$, there is no self-orthogonal $8$-$(24m,4m+4,\lambda_{8})$ design.

Then, for the design $D$,  
we have the following theorem.

\begin{Thm}\label{thm:support 8-design}
$D$ is never an $8$-design.
\end{Thm}

Thus the proof of Theorem \ref{thm:main thm} is completed.

By the Assmus-Mattson theorem, the support design of minimum weight of 
an extremal binary doubly even $[24m +8, 12m + 4, 4m + 4]$, 
respectively $[24m + 16, 12m + 8, 4m + 4]$, 
code is a $3$-design, $1$-design, respectively. 
We give the following results by a similar argument to the above. 

\begin{Thm}\label{cor:sub results} 
Let $D_{1}$ and $D_{2}$ be the support $t$-designs of the minimum weight of an extremal binary doubly even self-dual 
$[24m+8,12m+4,4m+4]$ code ($m \leq 158$) and $[24m+16,12m+8,4m+4]$ code ($m \leq 163$), respectively.

\begin{enumerate}
\item[$(1)$] 
If $D_{1}$ becomes a $4$-design, then $D_{1}$ is a $5$-design and $m$ must be in the set \\
$\{15,35,45,58,75,85,90,95,113,115,120,125 \}$. 
If $D_{1}$ becomes a $6$-design, then $m$ must be in the set $\{58,90,113 \}$. 
If $D_{1}$ becomes a $7$-design, then $m$ must be in the set $\{58\}$, 
and $D_{1}$ is never an $8$-design.

\item[$(2)$] 
If $D_{2}$ becomes a $2$-design, then $D_{2}$ is a $3$-design and $m$ must be in the set \\
$\{5$, $10$, $20$, $23$, $25$, $35$, $44$, $45$, $50$, $55$, 
$60$, $70$, $72$, $75$, $79$, 
$80$, $85$, $93$, $95$, $110$, $118$, $120$, 
$121$, $123$, $125$, $130$, $142$, $144$, 
$145$, $149$, $150$, $155$, $156$, $157$, 
$160$, $163$$\}$. 
If $D_{2}$ becomes a $4$-design, then $m$ must be in the set 
$\{10$, $79$, 
$93$, $118$, $120$, 
$123$, $125$, $142$$\}$. 
If $D_{2}$ becomes a $5$-design, then $m$ must be in the set 
$\{79$, 
$93$, $118$, $120$, 
$123$, $125$, $142$$\}$, 
and $D_{2}$ is never a $6$-design.

\end{enumerate}
\end{Thm}




\end{document}